\documentclass[12pt]{amsart}
\usepackage{amssymb}
\newtheorem*{thm}{Theorem}
\newtheorem*{lem}{Lemma}
\DeclareSymbolFont{script}{U}{eus}{m}{n}
\DeclareMathSymbol{\Wedge}{0}{script}{"5E}

\title[Beltrami via parabolic]
{Beltrami's theorem via parabolic geometry}
\author[Michael Eastwood]{Michael Eastwood}
\address{\hskip-\parindent
School of Mathematical Sciences\\
University of Adelaide\\ 
SA 5005\\ 
Australia}
\email{meastwoo@member.ams.org}

\thanks{This work was also supported by the Simons Foundation grant 346300 and
the Polish Government MNiSW 2015--2019 matching fund. It was written whilst the
author was at the Banach Centre at IMPAN in Warsaw for the Simons Semester
`Symmetry and Geometric Structures.'}

\begin{document}
\begin{abstract} 
We use Beltrami's theorem as an excuse to present some arguments from parabolic
differential geometry without any of the parabolic machinery.
\end{abstract}    
\maketitle 

\setcounter{section}{-1}
\section{Introduction}
One version of Beltrami's theorem~\cite{B} is as follows.
\begin{thm} Suppose $M$ is a smooth two-dimensional manifold and $g_{ab}$ is a
Riemannian metric on~$M$. Then $g_{ab}$ is constant curvature if and only if
there are local co\"ordinates near any point in which the geodesics of $g_{ab}$
become straight lines.
\end{thm}    

Antonio~Di~Scala~\cite{dS} has a very nice proof of this theorem, which he
kindly explained to me (in English). He asked me about a proof by the methods
of parabolic geometry but insisted that I did not wave my hands at all. In
particular, I was not allowed to say `Cartan connection' nor `development map.'

So what follows is an application of some key ideas from parabolic geometry
without actually explaining any of the underlying theory. The discerning reader
will correctly suspect that this is just the tip of an iceberg. For a
comprehensive description of the iceberg itself, the reader is directed
to~\cite{parabook}.

I would like to thank Antonio~Di~Scala for many interesting discussions during 
the preparation of this article.

\section{Geodesics}\label{geo}
We need some notation and basic results on geodesics. Let $M$ be a smooth
two-dimensional manifold. We shall denote by $TM$ the tangent bundle of $M$
and by $\Wedge^1$ the bundle of $1$-forms on~$M$. Suppose
$$\nabla_a:TM\to\Wedge^1\otimes TM$$
is a torsion-free connection and $t\mapsto\gamma(t)\in M$ is a smooth curve.
Let us write $U^a$ for the velocity field along $\gamma$ and, having in mind a
torsion-free connection $\nabla_a$, write $\partial\equiv U^a\nabla_a$ for the
directional derivative along~$\gamma$. Then $\gamma$ is an affinely
parameterised geodesic for $\nabla_a$ if and only if the acceleration field
$\partial U^a$ vanishes. The geodesics in Beltrami's theorem, however, are 
{\em unparameterised\/} curves. It means that we should instead allow 
$\partial U^a= fU^a$ for some smooth function~$f$.

\begin{lem} In order that two torsion-free connections $\nabla_a$ and\/
$\widehat\nabla_a$ have the same unparameterised geodesics, it is necessary 
and sufficient that locally 
\begin{equation}\label{proj}
\widehat\nabla_aX^c=\nabla_aX^c+\Upsilon_aX^c+\Upsilon_bX^b\delta_a{}^c,
\end{equation}
where $\delta_a{}^c$ denotes the identity tensor and\/ $\Upsilon_a$ is an 
arbitrary $1$-form.
\end{lem}
\begin{proof} The general formula relating two torsion-free connections is
$$\widehat\nabla_aX^c=\nabla_aX^c+\Gamma_{ab}{}^cX^b,\quad\mbox{where}\enskip
\Gamma_{ab}{}^c=\Gamma_{ba}{}^c.$$
Therefore $\widehat\partial U^c=\partial U^c+U^a\Gamma_{ab}{}^cU^b$ and so we
require that
$$U^aU^b\Gamma_{ab}{}^c\propto U^c\quad\mbox{for all}\enskip U^a.$$
It is a matter of linear algebra to check that this forces 
$$\Gamma_{ab}{}^c=\delta_a{}^c\Upsilon_b+\delta_b{}^c\Upsilon_a$$
for some $\Upsilon_a$.
\end{proof}
The Ricci tensor $R_{ab}$ of $\nabla_a$, characterised by
$$(\nabla_a\nabla_b-\nabla_b\nabla_a)X^b=-R_{ab}X^b$$
for all vector fields $X^b$, is not necessarily symmetric. However, it is easy
to check that if $\widehat\nabla_a$ and $\nabla_a$ are related by~(\ref{proj}),
then
\begin{equation}\label{Ricci_change}\widehat 
R_{ab}=R_{ab}-2\nabla_a\Upsilon_b+\nabla_b\Upsilon_a+\Upsilon_a\Upsilon_b.
\end{equation}
Locally, therefore, we may always use (\ref{proj}) to arrange that the Ricci 
tensor be symmetric without disturbing its unparameterised geodesics.

Instead of proving Beltrami's theorem itself, we shall establish a more general
result concerning the geodesics of an arbitrary torsion-free affine 
connection. Bearing in mind that we can arrange the Ricci tensor to be 
symmetric, the more general result can be stated as follows.

\begin{thm} Suppose $M$ is a smooth two-dimensional manifold. Denote by $TM$
the tangent bundle of $M$ and by\/ $\Wedge^1$ the bundle of $1$-forms on~$M$.
Suppose $\nabla_a:TM\to\Wedge^1\otimes TM$ is a torsion-free connection.
Suppose that its Ricci tensor $R_{ab}$ is symmetric. Let\/
$Y_{abc}\equiv\nabla_aR_{bc}-\nabla_bR_{ac}$. Then $Y_{abc}=0$ if and only if
there are local co\"ordinates near any point in which the geodesics of
$\nabla_a$ become straight lines.
\end{thm}

%%%%%%%%%%%%%%%%%%%%%%%%%%%%%%%
% Curvature conventions
% $$(\nabla_a\nabla_b-\nabla_b\nabla_a)X^c=R_{ab}{}^c{}_dX^d$$
% and in two dimensions
% $$R_{ab}{}^c{}_d
% =\delta_a{}^c\Rho_{bd}-\delta_b{}^c\Rho_{ad}+\beta_{ab}\delta_d{}^c,
% \quad\mbox{where}\enskip\beta_{ab}=-2\Rho_{[ab]}.$$
% In other words
% $$R_{ab}{}^c{}_d
% =\delta_a{}^c\Rho_{bd}-\delta_b{}^c\Rho_{ad}
% -\Rho_{ab}\delta_d{}^c+\Rho_{ba}\delta_d{}^c.$$
% In yet other words
% $$(\nabla_a\nabla_b-\nabla_b\nabla_a)X^c
% =\delta_a{}^c\Rho_{bd}X^d-\delta_b{}^c\Rho_{ad}X^d
% -\Rho_{ab}X^c+\Rho_{ba}X^c.$$
% Compare to Ricci:
% $$R_{bd}=2\Rho_{bd}-\Rho_{db}.$$
% Conversely,
% $$R_{bd}=\Rho_{bd}+2\Rho_{[bd]}\enskip\mbox{so}\enskip
% R_{[bd]}=3\Rho_{[bd]}$$
% and, therefore,
% $$\textstyle \Rho_{ab}=R_{ab}-2\Rho_{[ab]}=R_{ab}-\frac23R_{[ab]}
% =\frac23R_{ab}+\frac13R_{ba}$$
% Projective change is conveniently expressed as
% $$\widehat\Rho_{ab}=\Rho_{ab}-\nabla_a\Upsilon_a+\Upsilon_a\Upsilon_b.$$
% Therefore
% $$\widehat 
% R_{ab}=R_{ab}-2\nabla_a\Upsilon_b+\nabla_b\Upsilon_a+\Upsilon_a\Upsilon_b.$$
%%%%%%%%%%%%%%%%%%%%%%%%%%%%%%%

The reason that this theorem is more general than Beltrami's is that,
in case of a metric connection in two dimensions, the Gau{\ss}ian curvature $K$
is characterised by $R_{ab}=Kg_{ab}$ whence 
$Y_{abc}=(\nabla_aK)g_{bc}-(\nabla_bK)g_{ac}$.

\medskip\noindent{\bf Remark}\enskip Another convenience of a having a
symmetric Ricci tensor in two dimensions is that, in this case,
\begin{equation}\label{curv}(\nabla_a\nabla_b-\nabla_b\nabla_a)X^c
=\delta_a{}^cR_{bd}X^d-\delta_b{}^cR_{ad}X^d,\end{equation} 
as one may readily verify.

\section{A surprising connection}
Fixing a torsion-free connection $\nabla_a$ on $TM$ with symmetric Ricci 
tensor~$R_{ab}$, we define a 
connection, also denoted by $\nabla_a$, on the bundle 
${\mathbb{T}}\equiv TM\oplus\Wedge^0$ by 
\begin{equation}\label{tractors}{\mathbb{T}}=
\begin{array}{c}TM\\[-4pt] \oplus\\[-2pt]\Wedge^0\end{array}\ni
\left[\!\begin{array}{c}X^b\\ \rho\end{array}\!\right]
\stackrel{\nabla_a}{\longmapsto}
\left[\!\begin{array}{c}\nabla_aX^b-\delta_a{}^b\rho\\ 
\nabla_a\rho+R_{ab}X^b\end{array}\!\right]\in\Wedge^1\otimes{\mathbb{T}}.
\end{equation}
We may compute its curvature:
$$\nabla_a\nabla_b\left[\!\begin{array}{c}X^c\\ \rho\end{array}\!\right]
% =\nabla_a\left[\!\begin{array}{c}\nabla_bX^c-\delta_b{}^c\rho\\
% \nabla_b\rho+R_{bc}X^c\end{array}\!\right]
=\left[\!\begin{array}{c}\nabla_a(\nabla_bX^c-\delta_b{}^c\rho)
-\delta_a{}^c(\nabla_b\rho+R_{bd}X^d)\\
\nabla_a(\nabla_b\rho+R_{bc}X^c)+R_{ac}(\nabla_bX^c-\delta_b{}^c\rho)
\end{array}\!\right]$$
so
$$(\nabla_a\nabla_b-\nabla_b\nabla_a)
\left[\!\begin{array}{c}X^c\\ \rho\end{array}\!\right]
=\left[\!\begin{array}{c}(\nabla_a\nabla_b-\nabla_b\nabla_a)X^c-
(\delta_a{}^cR_{bd}-\delta_b{}^cR_{ad})X^d\\
(\nabla_aR_{bc}-\nabla_bR_{ac})X^c
\end{array}\!\right]$$
but, according to (\ref{curv}), the first row vanishes and so we are left with
$$(\nabla_a\nabla_b-\nabla_b\nabla_a)
\left[\!\begin{array}{c}X^c\\ \rho\end{array}\!\right]
=\left[\!\begin{array}{c}0\\
Y_{abc}X^c
\end{array}\!\right].$$
In other words, this connection is flat if and only if $Y_{abc}=0$, which is 
somewhat surprising.

\section{Proof of main theorem}
We are now in a position to prove the generalised Beltrami theorem from
Section~\ref{geo}. One direction is mindless computation. Specifically, if the
geodesics of $\nabla_a$ are straight lines in local co\"ordinates, then
(\ref{proj}) holds with $\widehat\nabla_a$ being flat. According to
(\ref{Ricci_change}) with $R_{ab}$ symmetric, we conclude that
$$R_{ab}=\nabla_a\Upsilon_b-\Upsilon_a\Upsilon_b\quad\mbox{and}\quad
\nabla_{[a}\Upsilon_{b]}=0.$$
We now compute
$$Y_{abd}=\nabla_aR_{bd}-\nabla_bR_{ad}
=(\nabla_a\nabla_b-\nabla_b\nabla_a)\Upsilon_d
-\Upsilon_b\nabla_a\Upsilon_d+\Upsilon_a\nabla_b\Upsilon_d$$
and, from~(\ref{curv}), conclude that
$$Y_{abd}
=-\Upsilon_a(R_{bd}-\nabla_b\Upsilon_d)+\Upsilon_b(R_{ad}-\nabla_a\Upsilon_d)
=-\Upsilon_a\Upsilon_b\Upsilon_d+\Upsilon_b\Upsilon_a\Upsilon_d,$$
which vanishes, as advertised.

In other other direction we use the surprising flat
connection~(\ref{tractors}). As the bundle ${\mathbb{T}}$ has rank~$3$, locally
we may find a three-dimensional space of covariant constant sections, which we
shall identify as~${\mathbb{R}}^3$ (and replace $M$ by a suitable open subset
on which this is valid). Each point $x\in M$ now gives rise to a
$1$-dimensional linear subspace in~${\mathbb{R}}^3$, namely
$$L_x\equiv\left\{\left[\!\begin{array}{c}X^b\\ \rho\end{array}\!\right]\in
\Gamma({\mathbb{T}})
\mbox{ s.t.\ }
\nabla_a\left[\!\begin{array}{c}X^b\\ \rho\end{array}\!\right]=0
\enskip\mbox{and}\enskip X^b|_x=0\right\}.$$
In this way, we obtain $\phi:M\looparrowright{\mathbb{RP}}_2$ (and replace $M$
by a suitable open subset on which $\phi:M\hookrightarrow{\mathbb{RP}}_2$).
Now suppose $\gamma\hookrightarrow M$ is a geodesic with velocity 
vector~$U^a$, as before. Restricting the connection (\ref{tractors}) to 
$\gamma$ gives $\partial:{\mathbb{T}}\to{\mathbb{T}}$. Specifically,
\begin{equation}\label{tractors_along_gamma}
\partial\left[\!\begin{array}{c}X^b\\ \rho\end{array}\!\right]=
\left[\!\begin{array}{c}\partial X^b-\rho U^b\\ 
\partial\rho+U^aR_{ab}X^b\end{array}\!\right].\end{equation}
In particular, if $X^b$ is somewhere is somewhere tangent to~$\gamma$, then 
this is always the case along $\gamma$ and (\ref{tractors_along_gamma}) becomes
$$\partial\left[\!\begin{array}{c}fU^b\\ \rho\end{array}\!\right]=
\left[\!\begin{array}{c}(\partial f-\rho)U^b\\ 
\partial\rho+fU^aR_{ab}U^b\end{array}\!\right],$$
which is simply the second order linear ordinary differential equation
$$\partial\partial f+R_{ab}U^aU^b f=0$$
along~$\gamma$. Its two-dimensional space of solutions is a linear subspace of
the space of covariant constant sections of~${\mathbb{T}}$. So the geodesic
$\gamma$ gives a straight line in~${\mathbb{RP}}_2$. Otherwise said, the
diffeomorphism $\phi:M\hookrightarrow{\mathbb{RP}}_2$ maps geodesics in $M$ to
straight lines in~${\mathbb{RP}}_2$. These lines may be viewed in a
standard affine chart ${\mathbb{R}}^2\subset{\mathbb{RP}}_2$ and the proof is
complete.

\section{Beltrami's task}
In fact, Antonio Di Scala pointed out to me that Beltrami finds in \cite{B} the
general form of a metric on ${\mathbb{R}}^2$ with the property that its
geodesics, as unparameterised curves, are straight lines and, only having done
this, does he note that these metrics have constant Gau{\ss}ian curvature. This
is a much more challenging task but one that is also familiar from parabolic
differential geometry (as a particular instance of finding solutions to a
so-called `first BGG operator').

Without going into details, observations of R.~Liouville~\cite{L} combine with
the connection (\ref{tractors}) on ${\mathbb{R}}^2$ in allowing one easily to
write down the general metric defined on
$U^{\mathrm{open}}\subseteq{\mathbb{R}}^2$ and having the property that its
geodesics are straight lines. There is a six parameter family thereof:
$$\frac{(ry^2+2qy+u)dx^2-2(rxy+qx+py+t)dx\,dy+(rx^2+2px+s)dy^2}
{\big((rx^2+2px+s)(ry^2+2qy+u)-(rxy+qx+py+t)^2\big)^2}$$
for $p,q,r,s,t,u\in{\mathbb{R}}$, defined wherever this expression is positive
definite.

The case $(p,q,r,s,t,u)=(0,0,1,1,0,1)$ gives the Thales metric
$$\frac{(1+y^2)\,dx^2-2xy\,dx\,dy+(1+x^2)\,dy^2}{(1+x^2+y^2)^2}$$
defined everywhere, whilst the case $(p,q,r,s,t,u)=(0,0,-1,1,0,1)$ is the 
Beltrami metric
$$\frac{(1-y^2)\,dx^2+2xy\,dx\,dy+(1-x^2)\,dy^2}{(1-x^2-y^2)^2}$$
defined on the unit disc.

In any case, a rather involved computation confirms that the Gau{\ss}ian 
curvature is constant in the general case, specifically 
$$K=r(su-t^2)-p^2u+2pqt-q^2s.$$

\end{document}